\documentclass[11pt]{amsart}
\usepackage[top=1.5in, bottom=1in, left=0.9in, right=0.9in]{geometry}

\usepackage{amscd,amsmath,amssymb,fancyhdr,color}
\usepackage[utf8]{inputenc}
\usepackage{amsfonts}
\usepackage{amsthm}
\usepackage{accents}
\usepackage{graphicx}
\usepackage{float}
\usepackage{subcaption}
\usepackage{verbatim}
\usepackage{indentfirst}
\usepackage{dsfont}
\usepackage{tikz}
\usepackage{tikz-cd}
\usetikzlibrary{matrix}
\usepackage[all]{xy}
\usepackage{enumerate}
\usepackage{extarrows}
\usepackage{csquotes}

\usepackage{scalerel,stackengine}
\stackMath
\newcommand\reallywidehat[1]{%
	\savestack{\tmpbox}{\stretchto{%
			\scaleto{%
				\scalerel*[\widthof{\ensuremath{#1}}]{\kern-.6pt\bigwedge\kern-.6pt}%
				{\rule[-\textheight/2]{1ex}{\textheight}}
			}{\textheight}%
		}{0.5ex}}%
	\stackon[1pt]{#1}{\tmpbox}%
}

\usepackage{faktor}

\usepackage{BOONDOX-uprscr}

\usepackage[backref=page]{hyperref}
\renewcommand*{\backref}[1]{}
\renewcommand*{\backrefalt}[4]{%
	\ifcase #1 (Not cited.)%
	\or        (Cited on page~#2.)%
	\else      (Cited on pages~#2.)%
	\fi}

\hypersetup{
	colorlinks   = true,
	citecolor    = magenta
}

\DeclareMathOperator{\Ann}{Ann}

\DeclareMathOperator{\Supp}{Supp}

\DeclareMathOperator{\CL}{CL}

\numberwithin{equation}{section}

\def\eqref#1{(\ref{#1})}

\newcommand{\N}{{\mathbb N}}
\newcommand{\Z}{{\mathbb Z}}
\newcommand{\C}{{\mathbb C}}

\def\1{\sqrt{-1}\:}

\newcommand{\cntrct}                
{\hspace{2pt}\raisebox{1pt}{\text{$\lrcorner$}}\hspace{2pt}}

\newcommand{\ie}{{\em i.e. }}

\renewcommand{\to}{\longrightarrow}

\newcounter{Mycounter}[section]
\newcounter{lemma}[section]
\newcounter{claim}[section]
\newcounter{sublemma}[section]
\newcounter{corollary}[section]
\newcounter{theorem}[section]
\newcounter{conjecture}[section]
\newcounter{proposition}[section]
\newcounter{definition}[section]
\newcounter{example}[section]
\newcounter{remark}[section]
\newcounter{problem}[section]
\newcounter{question}[section]
\makeatletter

\@addtoreset{equation}{section}

\@addtoreset{footnote}{section}

\makeatother

\usetikzlibrary{arrows,chains,matrix,positioning,scopes}

\makeatletter
\tikzset{join/.code=\tikzset{after node path={%
			\ifx\tikzchainprevious\pgfutil@empty\else(\tikzchainprevious)%
			edge[every join]#1(\tikzchaincurrent)\fi}}}
\makeatother

\tikzset{>=stealth',every on chain/.append style={join},
	every join/.style={->}}

\makeatletter
\newtheorem*{rep@theorem}{\rep@title}
\newcommand{\newreptheorem}[2]{%
	\newenvironment{rep#1}[1]{%
		\def\rep@title{\ref{##1}}%
		\begin{rep@theorem}}%
		{\end{rep@theorem}}}
\makeatother

\newreptheorem{theorem}{Theorem}

\begin{document}
	
	\newpage
	
	\title[The Cœuré--Loeb example as an open subset of a Stein space]{The Cœuré--Loeb example as an open subset of a Stein space}

	\author{Ovidiu Preda}
	\address{Ovidiu Preda \newline
		\textsc{\indent University of Bucharest, Faculty of Mathematics and Computer Science\newline 
			\indent 14 Academiei Str., Bucharest, Romania\newline
			\indent \indent and\newline
			\indent Institute of Mathematics ``Simion Stoilow'' of the Romanian Academy\newline 
			\indent 21 Calea Grivitei Street, 010702, Bucharest, Romania}}
	\email{ovidiu.preda@fmi.unibuc.ro; ovidiu.preda@imar.ro}

	\thanks{     \\[.1cm]
	}
	
	\date{\today}

	\begin{abstract}
		The Serre problem asked if a locally trivial holomorphic fibration with Stein base and Stein fiber is itself Stein. \v Skoda constructed the first counterexample for this problem. Later, Cœuré and Loeb constructed another remarkable counterexample, with bounded domain of holomorphy as fiber. The total space of their fibration has global holomorphic functions which separate points and provide local coordinates. However, it does not admit a Stein envelope of holomorphy. In this paper, we prove that the Cœuré--Loeb fibration can be realized as an open subset of a Stein space, more precisely it is the complement of an analytic set in a Stein space. 
	\end{abstract}
	
	\maketitle
	
	\hypersetup{linkcolor=blue}
	\tableofcontents

	\section{Introduction}
	
	Stein spaces are complex spaces which are holomorphically convex and on which global holomorphic functions separate points and give local coordinates. Open subsets of Stein spaces always satisfy the latter two conditions. Thus, a natural question is whether any complex space on which global holomorphic functions separate points and give local coordinates can be realized as an open subset of a Stein space. This problem was studied by Col\c toiu and Joi\c ta \cite{coltoiu_joita}, who constructed two counterexamples, both of which are complex spaces which are locally reducible. The problem of finding a smooth counterexample for the open immersion problem remained open. 
	
	Cœuré and Loeb \cite{coeure_loeb} constructed a locally trivial holomorphic fibration with $\C^*$ as base and a bounded domain of holomorphy in $\C^2$ as fiber, such that the total space is not Stein, thus giving a remarkable new counterexample to the Serre problem, which asks whether a locally trivial holomorphic fibration with Stein base and Stein fiber is always Stein. Zaffran \cite{zaffran2001} generalized the Cœuré--Loeb counterexample to a class of similar fibrations, and showed that they do not admit a Stein envelope of holomorphy. Preda \cite{preda2015} gave a different and much shorter proof of the non-existence of the envelope of holomorphy. Siu’s argument \cite[p.499]{Siu78} implies that global holomorphic functions on the Cœuré--Loeb fibration separate points and provide local coordinates. Col\c toiu \cite[Problem 7]{coltoiu_open_pb}, in connection with the search for a smooth counterexample for the open immersion problem, asked whether the Cœuré--Loeb fibration is isomorphic to an open subset of a Stein space, as it seemed to be a good candidate.
	
	In this article, we prove that the answer is positive; hence, the Cœuré--Loeb fibration is not a counterexample to the open immersion problem in the smooth case.
	
	\medskip
	
	The paper is organized as follows: Section \ref{sec:prelim} contains some important definitions, results, and standard constructions that we need throughout this paper, and Section \ref{sec:main} contains the proof of our main result, split into two parts and structured into several steps, to make it easier to read. It ends with a remark about another problem from Col\c toiu's list \cite{coltoiu_open_pb} related to the Cœuré--Loeb fibration.
	
	\section{Preliminaries}\label{sec:prelim}
	
	Throughout the paper, complex spaces are assumed to be reduced and countable at infinity, and all maps between them are holomorphic, even if not mentioned explicitly, unless stated otherwise. For a complex space $X$, we write $\mathcal O(X):=\Gamma(X,\mathcal O_X)$ for the ring of global holomorphic functions. 
	
	Let $X$ be a complex space and $x\in X$. We say that the functions $f_1,\ldots,f_N\in\mathcal{O}(X)$ \emph{provide local coordinates} at $x$ if the map $F=(f_1,\ldots,f_N):X\rightarrow\C^N$ is a local embedding at $x$. Equivalently, the induced ring morphism $$F_x^{\#}:\mathcal{O}_{\C^N, F(x)}\rightarrow \mathcal{O}_{X,x}$$ is surjective. Another equivalent formulation is that the germs of $f_1-f_1(x),\ldots,f_N-f_N(x)$ at $x$ must span $\mathfrak{m}_x / \mathfrak{m}_x^2$, where $\mathfrak{m}_x$ is the maximal ideal of $\mathcal{O}_{X,x}$. 
	
	An \emph{open immersion} is a holomorphic map which identifies its source biholomorphically with an open complex subspace of its target. 
	
	Let $X$ and $\widetilde X$ be complex spaces. A \emph{modification} is a
	proper surjective holomorphic map $\rho\colon \widetilde X\longrightarrow X$ for which there exists a nowhere-dense closed analytic subset $A\subset X$ such that $\rho^{-1}(A)$ is nowhere dense in $\widetilde X$ and the restriction
	\[
	\rho\colon \widetilde X\setminus\rho^{-1}(A)
	\longrightarrow X\setminus A
	\]
	is a biholomorphism. The set $A$ is called a \emph{center} of the modification,
	and $\rho^{-1}(A)$ is the corresponding \emph{exceptional set}.
	
	A \emph{Stein envelope of holomorphy} of $X$ is a holomorphic map $\alpha\colon X\to\widehat X$, where $\widehat X$ is Stein, such that $\alpha^*\colon\mathcal O(\widehat X)\to\mathcal O(X)$ is an isomorphism and, for every holomorphic map $\varphi\colon X\to Z$ for which $\varphi^*\colon\mathcal O(Z)\to\mathcal O(X)$ is an isomorphism, there exists a unique holomorphic map $\beta\colon Z\to\widehat X$ satisfying $\alpha=\beta\circ\varphi$.
	
	\subsection{Principal additive bundles}
	
	A \emph{principal holomorphic $(\C,+)$-bundle} is a holomorphic map $q:P\rightarrow B$ equipped with a holomorphic $(\C,+)$-action such that, over every member $U$ of a suitable open cover of $B$, it is equivariantly biholomorphic to $U\times\C\rightarrow U$, with the action $t\cdot(x,s)=(x,s+t)$. Choose a trivializing cover $\mathcal U=\{U_i\}$ and fiber coordinates $s_i$ on $q^{-1}(U_i)$. On an overlap they satisfy 
	$$s_i=s_j+c_{ij}, \qquad c_{ij}\in\mathcal O(U_i\cap U_j),$$ 
	where $(c_{ij})$ is a \v{C}ech $1$-cocycle. Its class $\xi(P):=[(c_{ij})]\in H^1(B,\mathcal O_B)$ is the \emph{torsor class} of $P$. Changing the trivializations changes $(c_{ij})$ by a coboundary, and the resulting correspondence identifies isomorphism classes of principal holomorphic $(\C,+)$-bundles over $B$ with $H^1(B,\mathcal O_B)$.
	
	The group $H^1(B,\mathcal O_B)$ is naturally an $\mathcal O(B)$-module. If $f\in\mathcal O(B)$, the class $f\xi(P)$ is represented by $(fc_{ij})$. Consequently, if $f\xi(P)=0$, then, after refining $\mathcal U$ if necessary, there are functions $b_i\in\mathcal O(U_i)$ such that
	$b_i-b_j=fc_{ij}$ on $U_i\cap U_j$. Here we use the standard identification of sheaf cohomology with the direct limit of \v{C}ech cohomology over refinements.
	
	We shall also use the following elementary quotient construction. Suppose that a discrete group $\Gamma$ acts freely and properly discontinuously on a complex space $\widetilde B$, and let $\lambda:\Gamma\to(\C,+)$ be a homomorphism. The diagonal action $\gamma\cdot(x,z)=\bigl(\gamma x,z+\lambda(\gamma)\bigr)$ on $\widetilde B\times\C$ is free and properly discontinuous, and $(\widetilde B\times\C)/\Gamma \longrightarrow\widetilde B/\Gamma$ is a principal holomorphic $(\C,+)$-bundle. Its additive action is $t\cdot[x,z]=[x,z+t]$.
	
	\subsection{Negative-definite cycles and the cusp model}
	
	Let $C=\bigcup_{\nu=1}^{\ell}C_\nu$ be a compact curve in a smooth complex surface, where $C_\nu$ are the irreducible components of $C$. It is \emph{negative definite} if its intersection matrix $$(C_\mu\cdot C_\nu)_{1\leq\mu,\nu\leq\ell}$$
	is negative definite. A \emph{rational normal-crossing cycle} is a normal-crossing curve whose irreducible components are rational and whose
	dual graph is a cycle. Grauert's contraction theorem \cite{grauert} implies that a connected negative-definite curve can be contracted, in a neighborhood of the curve, to an isolated point. After normalization, the target is a normal surface singularity or a normal smooth surface germ. A normal surface singularity whose minimal resolution has a rational cycle or an irreducible rational curve with one node as exceptional divisor is called a \emph{cusp singularity}.
	
	We emphasize that negative definiteness gives a local analytic contraction; by itself, it does not imply that the contraction of an arbitrary noncompact surface is globally Stein. In the special model used below, global Steinness comes from the Hirzebruch--Pinkham cusp construction.
	
	We finally summarize the geometric input from \cite[Sections 2, 6, 7]{zaffran2001}. For the Dloussky matrix $A=\bigl( \begin{smallmatrix} 1 & 1 \\ 1 & 2 \end{smallmatrix} \bigr)$, let $R^d\subset(\C^*)^2$ be Zaffran's invariant Reinhardt domain and let $F(u,v)=(uv,uv^2)$. The automorphism $F$ generates a free and properly discontinuous $\Z$-action on $R^d$. The Cœuré--Loeb example is $E_{\CL}=(R^d\times\C)/\Z$, where the generator acts by $((u,v),z)\longmapsto(F(u,v),z+1)$. Its projection to $R^d/\Z$ is a principal holomorphic $(\C,+)$-bundle.
	
	Zaffran constructs an equivariant partial compactification $\widehat R^d$. The $\Z$-action extends to $\widehat R^d$ and remains free and properly discontinuous, and $\widetilde Y:=\widehat R^d/\Z$ is the resolved cusp obtained from $R^d/\Z$ by adding a compact rational
	normal-crossing cycle $C$. The cycle is negative definite, and its contraction is the cusp point $a$ of a normal Stein surface $Y$; moreover, $\widetilde Y\setminus C \cong R^d/\Z \cong Y\setminus\{a\}$. This is the Hirzebruch--Pinkham Stein completion of the punctured cusp; see also \cite[p.307]{pinkham84}. Moreover,
	$$\widehat E_{\CL}=(\widehat R^d\times\C)/\Z\longrightarrow \widetilde Y$$
	is a principal holomorphic $(\C,+)$-bundle whose restriction over $\widetilde Y\setminus C$ is $E_{\CL}\to R^d/\Z$. Consequently, the torsor class of $E_{\CL}\to R^d/\Z$ is the restriction of the torsor class of $\widehat E_{\CL}\to\widetilde Y$.

	\section{The main result}\label{sec:main}
	
	First, we prove that a principal holomorphic $(\C,+)$-bundle over a punctured normal Stein surface can be realized as an open subset of a Stein space, under the assumption that the torsor class of the bundle is the restriction of a class coming from a modification of the base over the special point. 
	
	This is in the spirit of Hedén's work \cite{heden16}. He studied principal additive $\mathbb{G}_a$-bundles over a punctured normal affine surface $S_*=S\setminus \{\mathbf{x}\}$, and defined an affine extension as a normal affine $\mathbb{G}_a$-variety $X$ over $S$ containing the bundle as the preimage of $S_*$. This is, essentially, the algebraic analogue of our conclusion $P\cong X\setminus \pi^{-1}(a)$ of \ref{prop_extension_H1_class}. The main difference is that Hedén's context is algebraic and affine, assumes the puncture is a regular closed point, and asks for an extension carrying a compatible $\mathbb{G}_a$-action. Our \ref{prop_extension_H1_class} is analytic, allows a normal Stein surface and a modification, and only concludes that the bundle is an open subset of a Stein space, not necessarily Stein itself. 
	
	Brenner's article \cite{brenner13} is also relevant for the construction of a principal bundle over a punctured surface. He explains that a cohomology class on an open set determines an affine-linear bundle, and that a forcing algebra gives a natural completion of the bundle. Brenner's work also contains a warning particularly relevant here: in characteristic zero, over a two-dimensional normal local domain with a non-rational singularity, there are nontrivial $\mathbb{A}^1$-bundles over the punctured spectrum whose classes come from a resolution, but whose total spaces are not affine \cite[Thm.~6.8]{brenner13}. That is a reason why we cannot expect in our setting $P$ to be Stein, and we only prove that it is an open set in a Stein space. 
	
	Finally, Dubouloz, Hedén, and Kishimoto \cite[Sec. 3.3.1]{DuHeKi20} construct, for homogeneous $\mathbb{G}_a$-torsors over $\mathbb{A}_*^2$, an affine extension, and the original torsor is realized as the complement of the fiber above the puncture. This is a concrete algebraic model of the pattern $P=X\setminus \pi^{-1}(a)$ that appears in our result.

	\begin{theorem}\label{prop_extension_H1_class}
		Let $Y$ be a normal Stein surface, let $a\in Y$, denote $Y^*=Y\setminus \{a\}$, and let $\rho:\widetilde{Y}\rightarrow Y$ be a modification such that $\rho:\widetilde{Y}\setminus A \rightarrow Y\setminus\{a\}$ is a biholomorphism, where $A=\rho^{-1}(a)$. Let $q:P\rightarrow Y^*$ be a principal $(\C,+)$-bundle. Assume that the torsor class $\xi\in H^1(Y^*,\mathcal{O}_{Y^*})$ of the bundle $P$ is the restriction, via $\widetilde{Y}\setminus A\cong Y^*$, of a class $\tilde{\xi}\in H^1(\widetilde{Y},\mathcal{O}_{\widetilde{Y}})$. 
		
		Then $P$ is biholomorphic to an open subset of a Stein space. More precisely, there is a Stein space $X$ with a holomorphic map $\pi:X\to Y$ such that $P\cong X\setminus \pi^{-1}(a)$.
	\end{theorem}
	\begin{proof}
		For clarity, we organize the proof into several steps. 
		
		\vspace{5pt}
		
		\hspace{15pt} \textit{Step 1.} We prove that the class $\xi$ is killed by a sufficiently high power of the maximal ideal $\mathfrak{m}_a\subset \mathcal{O}_{Y,a}$. Since 
		$\rho:\widetilde{Y}\rightarrow Y$ is a proper modification and $Y$ is normal, we have $\rho_*\mathcal{O}_{\widetilde{Y}}=\mathcal{O}_Y$. The low-degree Leray exact sequence for $\rho$ gives 
		$$0\to H^1(Y,\rho_*\mathcal{O}_{\widetilde{Y}})\to H^1(\widetilde{Y},\mathcal{O}_{\widetilde{Y}})\to H^0(Y, R^1\rho_*\mathcal{O}_{\widetilde{Y}})\to H^2(Y,\rho_*\mathcal{O}_{\widetilde{Y}}).$$
		Since $Y$ is Stein, by Cartan's theorem B we get $H^1(Y,\mathcal{O}_Y)=H^2(Y,\mathcal{O}_Y)=0$. Hence, the exact sequence implies that $H^1(\widetilde{Y},\mathcal{O}_{\widetilde{Y}})\cong H^0(Y, R^1\rho_*\mathcal{O}_{\widetilde{Y}})$. Also, Grauert's direct image theorem ensures that $R^1\rho_*\mathcal{O}_{\widetilde{Y}}$ is a coherent $\mathcal{O}_Y$-module. Since $\rho$ is biholomorphic over $Y^*=Y\setminus \{a\}$, this coherent sheaf is supported at the single point $a$. 
		
		Let $\mathfrak{m}_a\subset R:=\mathcal{O}_{Y,a}$ be the maximal ideal and denote $M:=(R^1\rho_*\mathcal{O}_{\widetilde{Y}})_a$. Let $I:=\Ann_R(M)$. Since $M$ is finitely generated, we have 
		$\Supp_{R}(M)=V(\Ann_R(M))=V(I)$.
		However, here if $M\neq 0$, then $\Supp _R(M)=\{\mathfrak{m}_a\}$, so $V(I)=\{\mathfrak{m}_a\}$. This means that $\sqrt{I}=\mathfrak{m}_a$. Since $R$ is Noetherian, this implies that for a sufficiently large $N$, we have $\mathfrak{m}_a^N\subset I$. But $I=\Ann_R(M)$, so this means $\mathfrak{m}_a^N M=0$. If $M=0$, we may take $N=1$. 
		
		Now we return to the class $\tilde{\xi}\in H^1(\widetilde{Y},\mathcal{O}_{\widetilde{Y}})\cong H^0(Y, R^1\rho_*\mathcal{O}_{\widetilde{Y}})$. Because $R^1\rho_*\mathcal{O}_{\widetilde{Y}}$ is zero away from $a$, any $f\in\mathcal{O}(Y)$ whose germ $f_a$ belongs to $\mathfrak{m}_a^N$ annihilates $\xi$ after restriction to $Y^*$, hence we get $f\cdot \tilde{\xi}=0$ in $H^0(Y,R^1\rho_*\mathcal{O}_{\widetilde{Y}})$. Via the Leray identification, $f\cdot \tilde{\xi}=0$ in $H^1(\widetilde{Y},\mathcal{O}_{\widetilde{Y}})$. After restriction to $Y^*$, we get $f_{\restriction{Y^*}}\cdot \xi=0$ in $H^1(Y^*,\mathcal{O}_{Y^*})$. We can write this in a short way as $\mathfrak{m}_a^N\cdot \xi=0$. 
		
		\vspace{5pt}
			
		\hspace{15pt} \textit{Step 2.} We prove that the class $\xi$ is killed by finitely many functions in $\mathcal{O}(Y)$ whose common zero set is exactly $\{a\}$. The required finite family of functions is obtained by a standard application of Cartan's theorems A and B. Let $\mathcal{I}_{\{a\}}\subset\mathcal{O}_Y$ be the coherent ideal sheaf of the analytic set $\{a\}$. By Cartan's theorem A, the stalk $(\mathcal{I}_{\{a\}})_a=\mathfrak{m}_a$ is generated by germs of global sections of $\mathcal{I}_{\{a\}}$. Since this stalk is finitely generated, we may choose $g_1,\ldots,g_m\in H^0(Y,\mathcal{I}_{\{a\}})\subset\mathcal{O}(Y)$ such that $g_{1,a},\ldots,g_{m,a}$ generate $\mathfrak{m}_a$. Next, let
		$$\mathcal{J}:=g_1\mathcal{O}_Y+\cdots+g_m\mathcal{O}_Y,$$
		and let $Z\subset Y$ be the closed complex subspace defined by $\mathcal{J}$. We have $\mathcal{J}_a=\mathfrak{m}_a=(\mathcal{I}_{\{a\}})_a$. Since $\mathcal{I}_{\{a\}}/\mathcal{J}$ is coherent and its stalk at $a$ vanishes, there is a neighborhood $U$ of $a$ on which $\mathcal{J}=\mathcal{I}_{\{a\}}$. Consequently, $Z\cap U=\{a\}$ as complex spaces. Thus, $\{a\}$ is an open-and-closed reduced component of $Z$. Define $\overline{g}\in\mathcal{O}(Z)$ by
		$$
		\overline{g}=0\quad\text{on }\{a\},
		\qquad
		\overline{g}=1\quad\text{on }Z\setminus\{a\}.
		$$
		This is a well-defined holomorphic function because $\{a\}$ is open and closed in $Z$. From the exact sequence
		$$0\longrightarrow\mathcal{J}\longrightarrow\mathcal{O}_Y
		\longrightarrow\mathcal{O}_Z\longrightarrow 0$$
		we get that $\mathcal{J}$ is coherent, and by Cartan's theorem B we obtain $H^1(Y,\mathcal{J})=0$. Hence, the restriction map $\mathcal{O}(Y)\longrightarrow\mathcal{O}(Z)$ is surjective. We may therefore choose $g_{m+1}\in\mathcal{O}(Y)$ whose restriction to $Z$ is $\overline{g}$. Then $(g_{m+1})_a\in\mathcal{J}_a=\mathfrak{m}_a$, whereas $g_{m+1}=1$ on $Z\setminus\{a\}$. It follows that $g_1,\ldots,g_{m+1}$ have the single common zero $a$.
		
		Set $r:=m+1$ and, for $1\leq\nu\leq r$, define $f_\nu:=g_\nu^N$. The common zero set of $f_1,\ldots,f_r$ is still exactly $\{a\}$, and every germ $f_{\nu,a}$ belongs to $\mathfrak{m}_a^N$. Hence, by \textit{Step 1},
		$$f_\nu\cdot\xi=0\quad\text{in }H^1(Y^*,\mathcal{O}_{Y^*}),\qquad 1\leq\nu\leq r.$$
		
		\vspace{5pt}
		
		\hspace{15pt} \textit{Step 3.} We construct finitely many functions in $\mathcal{O}(P)$ which are combinations of base functions and weighted fiber coordinates, and which are used in the next step to define a biholomorphic map of  $P$. To this end, we represent the torsor $q:P\rightarrow Y^*$ by a \v Cech cocycle $(c_{ij})\in Z^1(\mathcal{U},\mathcal{O}_{Y^*})$ on a trivializing open cover $\mathcal{U}=\{U_i\}$ of $Y^*$. Thus, on $q^{-1}(U_i)$ there are fiber coordinates $s_i$ which satisfy $s_i=s_j+c_{ij}$ on $q^{-1}(U_i\cap U_j)$. Since $f_\nu\cdot \xi=0$ and $[(c_{ij})]=\xi\in H^1(Y^*,\mathcal{O}_{Y^*})$, after refining the cover if necessary, there are holomorphic functions $b_{\nu,i}\in\mathcal{O}(U_i)$ such that $b_{\nu,i}-b_{\nu,j}=f_\nu c_{ij}$ on $U_i\cap U_j$. Next, define $F_\nu=f_\nu s_i-b_{\nu,i}$ on $q^{-1}(U_i)$. Here and below, functions on the base are identified with their pullbacks to $P$. These local expressions of $F_\nu$ agree on the intersections $q^{-1}(U_i\cap U_j)$. Indeed, we have 
		$$f_\nu s_i-b_{\nu,i}=f_\nu(s_j+c_{ij})-b_{\nu,i}=f_\nu s_j +f_\nu c_{ij}-b_{\nu,i}=f_\nu s_j-b_{\nu,j},$$
		so $F_\nu\in\mathcal{O}(P)$ is well-defined. Next, for $1\leq \nu, \mu\leq r$, we define on $U_i\subset Y^*$ the functions ${d_{\mu \nu}}_{\restriction U_i}=f_\mu b_{\nu, i}-f_\nu b_{\mu, i}$. This definition is also independent of $i$, because on overlaps we have 
		$$(f_\mu b_{\nu, i}-f_\nu b_{\mu, i})-(f_\mu b_{\nu, j}-f_\nu b_{\mu, j})=f_\mu(b_{\nu,i}-b_{\nu,j})-f_{\nu}(b_{\mu,i}-b_{\mu,j})=f_\mu f_\nu c_{ij}-f_\nu f_\mu c_{ij}=0.$$
		Hence, $d_{\mu \nu}\in\mathcal{O}(Y^*)$. Since $Y$ is normal and $\{a\}$ has codimension $2$ in the surface $Y$, there exists a unique extension $d_{\mu \nu}\in\mathcal{O}(Y)$. 
		
		\vspace{5pt}
		
		\hspace{15pt} \textit{Step 4.} We construct an open immersion of $P$ into a Stein space. Let $T_1,\ldots,T_r$ denote the standard coordinates on $\C^r$, and identify the functions $f_\nu$ and $d_{\mu\nu}$ on $Y$ with their pullbacks to $Y\times\C^r$. Let
		$$
		X=\bigl\{f_\mu T_\nu-f_\nu T_\mu+d_{\mu\nu}=0
		\mid 1\leq\mu,\nu\leq r\bigr\}
		\subset Y\times\C^r
		$$
		be the closed analytic subset defined by these equations, endowed with its reduced complex-space structure. Let $\pi:X\rightarrow Y$ be the restriction of the first projection and put $X_{\restriction Y^*}:=\pi^{-1}(Y^*)$. We shall prove that $X_{\restriction Y^*}$ is biholomorphic to $P$.
		
		Consider the holomorphic map $\Phi:P\rightarrow Y^*\times\C^r$ defined by
		$\Phi(p)=\bigl(q(p),F_1(p),\ldots,F_r(p)\bigr)$. We first verify that $\Phi$ takes its values in $X_{\restriction Y^*}$. If $q(p)\in U_i$, then $F_\nu(p)=f_\nu(q(p))s_i(p)-b_{\nu,i}(q(p))$. Suppressing the evaluation at $q(p)$ and using $d_{\mu\nu}=f_\mu b_{\nu,i}-f_\nu b_{\mu,i}$ on $U_i$, we obtain
		$$
		f_\mu F_\nu-f_\nu F_\mu+d_{\mu\nu}
		=f_\mu(f_\nu s_i-b_{\nu,i})
		-f_\nu(f_\mu s_i-b_{\mu,i})
		+d_{\mu\nu}=0.
		$$
		Thus, $\Phi(P)$ is contained in the common zero set defining $X$, so we may restrict the target space to $\Phi:P\rightarrow X_{\restriction Y^*}$.
		
		We now construct its inverse. For $1\leq\alpha\leq r$, let $V_\alpha:=\{y\in Y^*\mid f_\alpha(y)\neq0\}$. Since the functions $f_1,\ldots,f_r$ have no common zero on $Y^*$, the open sets $V_1,\ldots,V_r$ cover $Y^*$. Over $U_i$, the equations defining $X_{\restriction Y^*}$ can be rewritten as
		$$
		f_\mu(T_\nu+b_{\nu,i})
		=f_\nu(T_\mu+b_{\mu,i}),
		\qquad 1\leq\mu,\nu\leq r.
		$$
		
		On $\pi^{-1}(U_i\cap V_\alpha)$, define
		$\widehat{s}_{i,\alpha}:=(T_\alpha+b_{\alpha,i})/f_\alpha$. This is holomorphic because $f_\alpha$ does not vanish on $V_\alpha$. If both $f_\alpha$ and $f_\beta$ are nonzero, then the preceding equations, with $(\mu,\nu)=(\alpha,\beta)$, give
		$f_\alpha(T_\beta+b_{\beta,i})=f_\beta(T_\alpha+b_{\alpha,i})$. Dividing by $f_\alpha f_\beta$, we obtain $\widehat{s}_{i,\alpha}=\widehat{s}_{i,\beta}$. Hence, as $\alpha$ varies, these functions glue to a holomorphic function $\widehat{s}_i\in\mathcal{O}\bigl(\pi^{-1}(U_i)\bigr)$.
		
		Taking $\mu=\alpha$ in the equations above and dividing by $f_\alpha$, we obtain, for every $\nu$,
		$$
		T_\nu+b_{\nu,i}
		=f_\nu\frac{T_\alpha+b_{\alpha,i}}{f_\alpha}
		=f_\nu\widehat{s}_i,
		\qquad\text{and hence}\qquad
		T_\nu=f_\nu\widehat{s}_i-b_{\nu,i}.
		$$
		
		We next verify compatibility with respect to the index $i$. On $\pi^{-1}(U_i\cap U_j\cap V_\alpha)$, the relation $b_{\alpha,i}-b_{\alpha,j}=f_\alpha c_{ij}$ gives
		$$
		\widehat{s}_i-\widehat{s}_j
		=\frac{T_\alpha+b_{\alpha,i}}{f_\alpha}
		-\frac{T_\alpha+b_{\alpha,j}}{f_\alpha}
		=\frac{b_{\alpha,i}-b_{\alpha,j}}{f_\alpha}
		=c_{ij}.
		$$
		Since the sets $V_\alpha$ cover $Y^*$, it follows that $\widehat{s}_i=\widehat{s}_j+c_{ij}$ on $\pi^{-1}(U_i\cap U_j)$. This is exactly the transition law satisfied by the fiber coordinates $s_i$ of the torsor $P$.
		
		Let $\tau_i:q^{-1}(U_i)\rightarrow U_i\times\C$ be the chosen trivialization, defined by $\tau_i(p)=\bigl(q(p),s_i(p)\bigr)$. On $\pi^{-1}(U_i)$, define
		$\Psi_i(x):=\tau_i^{-1}\bigl(\pi(x),\widehat{s}_i(x)\bigr)$. The transition identity $\widehat{s}_i=\widehat{s}_j+c_{ij}$ shows that the points represented by $\bigl(\pi(x),\widehat{s}_i(x)\bigr)$ and $\bigl(\pi(x),\widehat{s}_j(x)\bigr)$ in the two trivializations are the same point of $P$. Therefore, the maps $\Psi_i$ agree on the overlaps and glue to a holomorphic map $\Psi:X_{\restriction Y^*}\rightarrow P$.
		
		It remains to verify that $\Phi$ and $\Psi$ are inverse to one another. Let $p\in q^{-1}(U_i)$, put $y:=q(p)$, and choose $\alpha$ such that $f_\alpha(y)\neq0$. Since the coordinate $T_\alpha$ of $\Phi(p)$ is $F_\alpha(p)$, we have
		$$
		\widehat{s}_i\bigl(\Phi(p)\bigr)
		=\frac{F_\alpha(p)+b_{\alpha,i}(y)}{f_\alpha(y)}
		=\frac{f_\alpha(y)s_i(p)-b_{\alpha,i}(y)+b_{\alpha,i}(y)}
		{f_\alpha(y)}
		=s_i(p).
		$$
		Thus, $\Psi$ recovers both the base point $y$ and the fiber coordinate $s_i(p)$, and consequently $\Psi\circ\Phi=\operatorname{id}_P$.
		
		Conversely, let $x=(y,T_1,\ldots,T_r)\in\pi^{-1}(U_i)$ and put $p:=\Psi(x)$. By definition, $q(p)=y$ and $s_i(p)=\widehat{s}_i(x)$. Using the formula $T_\nu=f_\nu\widehat{s}_i-b_{\nu,i}$ obtained above, we get, for every $\nu$,
		$$
		F_\nu(p)
		=f_\nu(y)s_i(p)-b_{\nu,i}(y)
		=f_\nu(y)\widehat{s}_i(x)-b_{\nu,i}(y)
		=T_\nu.
		$$
		Hence, $\Phi(p)=x$, so $\Phi\circ\Psi=\operatorname{id}_{X_{\restriction Y^*}}$. We conclude that $\Phi$ is a biholomorphism and that
		$$
		P\cong X_{\restriction Y^*}
		=X\setminus\pi^{-1}(a).
		$$
		
		Finally, $Y$ is Stein, and therefore $Y\times\C^r$ is Stein. Since $X$ is a closed analytic subspace of $Y\times\C^r$, it is Stein. Thus, $P\cong X\setminus\pi^{-1}(a)$ is an open subset of a Stein space.
	\end{proof}
	
	\vspace{10pt}
	
	Now, we prove that the Cœuré--Loeb fibration satisfies the conditions in \ref{prop_extension_H1_class}. The proof collects all the results relevant to us, proved by Zaffran \cite[Sec. 6 and 7]{zaffran2001} about extending the fibration to a partial compactification of the base. The preliminary discussion in Section~\ref{sec:prelim} may be checked by the reader; parts of it are repeated in the proof of the next theorem.

	\begin{theorem}\label{CL_open_immersion}
		The Cœuré--Loeb example $E_{\CL}$ is biholomorphic to an open subset of a Stein space. More accurately, there exists a Stein space $X$ and an analytic subset $Z\subset X$ such that $E_{\CL}\cong X\setminus Z$. 
	\end{theorem}
	\begin{proof}
		In \cite{zaffran2001}, Zaffran constructs, for a suitable Reinhardt domain $R^d\subset (\C^*)^2$, the quotient $E=(R^d\times \C)/\Z$, where the generator of $\Z$ acts by
		$$R^d\times\C\ni ((u,v),z)\mapsto (F(u,v),z+1),$$
		and $F(u,v)=(u^av^b,u^cv^d)$ is associated with the Dloussky matrix $A=\bigl( \begin{smallmatrix} a & b \\ c & d \end{smallmatrix} \bigr)\in GL_2(\Z)$ (for details on this, one should check \cite[Part 1, p.397]{zaffran2001}).
		Then, he records that $E$ has two fibrations,
		$$E\xrightarrow{R^d} \C/\Z\cong \C^* \quad \text{and} \quad E\xrightarrow{\C} R^d/\Z,$$
		where the fibers are indicated above the arrows. For $A=\bigl( \begin{smallmatrix} 1 & 1 \\ 1 & 2 \end{smallmatrix} \bigr)$, the first fibration is exactly the Cœuré--Loeb example \cite{coeure_loeb}, described in \cite[Sec.6, p.408]{zaffran2001}; we shall denote it by $E_{\CL}$. 
		We are interested only in the second fibration $q:E\rightarrow R^d/\Z$, which is a holomorphic principal $(\C,+)$-bundle. Indeed, addition in the second factor of $R^d\times \C$ descends to a free and transitive fiberwise $(\C,+)$-action on the fibration $E=(R^d\times \C)/\Z$, because the generator changes the second coordinate by an additive constant. 
		
		Zaffran \cite[Sec.7, p.414]{zaffran2001} also constructs a partial compactification $\widehat{R}^d=\psi^{-1}(R^d)\cup \tilde{C}^0$ by adding an infinite string of rational curves, extends the action of $\Z$ to $\widehat{R}^d$, and defines $\widehat{E}=(\widehat{R}^d\times \C)/\Z$. Moreover, $R^d/\Z \hookrightarrow \widehat{R}^d/\Z$, and $\widehat{R}^d/\Z$ is a resolved cusp obtained by adding a compact cycle of rational curves to $R^d/\Z$. 
		
		Next, denote $\widetilde{Y}:=\widehat{R}^d/\Z$ and $C:=\widetilde{Y}\setminus (R^d/\Z)$. Thus, $C$ is the compact cusp cycle. By construction \cite[pp.399--400 and 414--415]{zaffran2001}, this exceptional cycle $C$ is a connected finite union of rational curves with normal crossings with the intersection matrix negative definite, hence it contracts to an isolated normal surface singularity. By \cite[p.307]{pinkham84}, the contracted space is Stein. Thus, there exists a modification $\rho:\widetilde{Y}\rightarrow Y$ onto a normal Stein surface $Y$, with $\rho(C)=\{a\}$ and $\widetilde{Y}\setminus C \cong Y\setminus\{a\}$. Consequently, $Y^*:=Y\setminus\{a\}\cong R^d/\Z$. 
		
		The fibration $\widehat{E}\rightarrow\widehat{R}^d/\Z=\widetilde{Y}$ is also a principal $(\C,+)$-bundle, and its restriction to $Y^*$ is exactly $E\rightarrow Y^*$, which for the Cœuré--Loeb example is $E_{\CL}\rightarrow Y^*$. Therefore, the torsor class $\xi\in H^1(Y^*,\mathcal{O}_{Y^*})$ of $E_{\CL}\rightarrow Y^*$ is the restriction of the torsor class $\tilde{\xi}\in H^1(\widetilde{Y},\mathcal{O}_{\widetilde{Y}})$ of $\widehat{E}\rightarrow \widetilde{Y}$. Now, all the hypotheses of \ref{prop_extension_H1_class} are satisfied. Hence, there exists a Stein space $X$, a holomorphic map $\pi:X\rightarrow Y$, and a biholomorphism $E_{\CL}\cong X\setminus Z$, where $Z=\pi^{-1}(a)$. This proves the theorem.
	\end{proof}
	
	\vspace{15pt}
	
	We end this section with a remark on another problem from Col\c toiu's list \cite{coltoiu_open_pb}, involving the Cœuré--Loeb fibration. 
	
	\begin{remark}
		Let $B$ be a Stein manifold and $\pi:X\rightarrow B$ be a locally trivial holomorphic fibration whose fiber is also a Stein manifold. The Cœuré--Loeb example shows that $X$ need not be Stein.
		\cite[Problem 8]{coltoiu_open_pb} asks whether $X$ is always $p_5$-convex, \ie whether for every sequence of continuous functions $f_n:\overline{\Delta}\rightarrow X$, $n\in \N$, which are holomorphic on $\Delta$ and $\bigcup_{n\in \N}f_n(\partial \Delta)\Subset X$, it follows that $\bigcup_{n\in \N}f_n(\overline{\Delta})\Subset X$. This has a negative answer, as shown by Oeljeklaus and Zaffran \cite[pp.463--464]{OeZaff06} even before the problem list \cite{coltoiu_open_pb} was published. 
	\end{remark}
	
	\vspace{20pt}
	
	\textbf{Acknowledgments.} The author acknowledges the use of general-purpose large language models as an aid in documentation, bibliographic searches, and proof-checking.

\end{document}